\documentclass[12pt]{amsart}
\usepackage{amsmath,amssymb,amsbsy,amsfonts,amsthm,latexsym,
                                amsopn,amstext,amsxtra,euscript,amscd,hyperref,enumitem}

\newtheorem{lem}{Lemma}
\newtheorem{thm}{Theorem}

\renewcommand{\a}{\alpha}
\renewcommand{\b}{\beta}

\newcommand{\f}{\varphi}
\newcommand{\ve}{\varepsilon}

\newcommand{\mc}{\lfloor m^c \rfloor}
\newcommand{\ec}{\lfloor c \rfloor}
\newcommand{\fc}{\{c\}}
\newcommand{\dc}{\|c\|}

\newcommand{\abs}[1]{\left| #1 \right|}

\DeclareMathOperator{\e}{\mathrm{e}}
\renewcommand{\mod}{\operatorname{mod}}

\begin{document}

\title[On a question of Luca and Schinzel]{On a question of Luca and Schinzel over Segal-Piatetski-Shapiro sequences}

\author[J.-M. Deshouillers, M. Hassani, M. Nasiri-Zare]{
{\sc Jean-Marc ~Deshouillers *},\\
{\sc  Mehdi Hassani}\\
and\\
{\sc Mohammad Nasiri-Zare}
}

\keywords{Segal-Piatetski-Shapiro sequences, Euler's totient function, distribution modulo 1.\\ * corresponding author}

\address{J.-M. D.
\newline Institut de Math\'ematiques de Bordeaux, UMR 5251
\newline Universit\'e de Bordeaux , Bordeaux INP et CNRS
\newline 33405 TALENCE Cedex
\newline FRANCE
\newline \textit{Email:} \texttt{jean-marc.deshouillers@math.u-bordeaux.fr}
}
\address{M. H.
\newline Department of Mathematics
\newline University of Zanjan
\newline University Blvd, 45137-38791, ZANJAN
\newline IRAN
\newline \textit{Email:} \texttt{mehdi.hassani@znu.ac.ir}
}
\address{M. N.-Z.
\newline Department of Mathematics
\newline University of Zanjan
\newline University Blvd, 45137-38791, ZANJAN
\newline IRAN
\newline \textit{Email:} \texttt{mnasirizare@znu.ac.ir}
}
\pagenumbering{arabic}

\date{}

\maketitle

\begin{abstract}
We extend to Segal-Piatetski-Shapiro sequences previous results on the Luca-Schinzel question. Namely, we prove that for any real $c$ larger than $1$ the sequence $(\sum_{m\le n} \varphi(\lfloor m^c \rfloor) /\lfloor m^c \rfloor)_n$ is dense modulo $1$, where $\varphi$ denotes Euler's totient function.  The main part of the proof consists in showing that when $R$  is a large integer, the sequence of the residues of $\lfloor m^c \rfloor$ modulo $R$ contains blocks of consecutive values which are in an arithmetic progression.
\end{abstract}

\section{Introduction}

At the Czech-Slovak Number Theory Conference in Smolenice in August 2007, F. Luca asked whether the sequences of arithmetic and geometric means of the first values of the Euler totient function are uniformly distributed modulo 1; A. Schinzel asked whether the weaker statement that those sequences are dense modulo 1 was known; this question was positively answered in \cite{DL}. This opened the way to extensions considering different multiplicative functions with constant mean value, or mean values over different sequences of integers as well as the original question of Luca in special cases (cf. the bibliography). One of the latest results is that of the first and third named authors of this paper
\begin{thm}[Theorem 1 of \cite{DNZ}]\label{thmDNZ}
Let $\f$ denote the Euler function and $G$ be a non constant polynomial with integral coefficients and taking positive values at positive arguments. The sequence 
$$
\left(\sum_{m\leq n}\frac{\f(G(m))}{G(m)}\right)_{n\ge 1}
$$ 
is dense modulo 1.
\end{thm}

Here, we extend this result to Segal-Piatetski-Shapiro sequences. 

\begin{thm}\label{mainthm}
Let $\f$ denote the Euler function and $c$ be a real number number larger than $1$. The sequence 
$$
\left(\sum_{m\leq n}\frac{\f(\lfloor m^c \rfloor)}{\lfloor m^c \rfloor}\right)_{n\ge 1}
$$ 
is dense modulo 1.
\end{thm}

The case when $c$ is an integer is but a special case of Theorem \ref{thmDNZ}. In this paper, we only consider the case when $c >1$ is not an integer, i.e. the case of Segal - Piatetski-Shapiro sequences.\\

A way to tackle the Luca-Schinzel question on a sequence $(a(n))_n$  is to obtain the arithmetical independence of some consecutive values of the sequence $(a(m))_m$ so that for each given $\ve$, one can find $H$ consecutive values of the sequence $((\f(a(m+h))/a(m+h))_m$ which are all less than $\ve$ but the sum of which is larger than $1$. A natural idea would be to expect a complete arithmetical independence of the sequence $(\lfloor(m+1)^c \rfloor, \ldots, \lfloor(m+H)^c \rfloor))_m$. This is not the case: it is proved in \cite{DDMS} that the sequence of the residues of $\mc$ modulo $k$ is normal for no $k  \ge 2$. It is even shown that in the sequence of the residues of $\mc$ modulo $k$, there are finite blocks of consecutive values which do not occur. \\

In Section \ref{slinpol}, we are however going to prove that $\lfloor (m+h)^c \rfloor$ can be locally approximated by a linear polynomial, more precisely, denoting by $\|c\|$ the so-called "distance to the nearest integer" of the real number $c$, we have
\begin{thm}\label{tlinpol}
Let $\b = 2^{(c+2)}(c+2)^2/\dc$ and let $H$ be a positive integer. If $R$ is sufficiently large an integer, for any residue $r$ modulo $R$, there exists $m \le R^{\b}$ such that
\begin{equation}\label{elinpol}
\forall h \in [1, H] \colon \lfloor (m+h)^c \rfloor \equiv r + h \; (\mod R).
\end{equation}
\end{thm}
In connection with \cite{DDMS}, we notice that this result implies that for any positive integer $H$ and any non integral $c$ larger than $1$, if $R$ is large enough, the sequence of the residues of $\lfloor n^c \rfloor$ modulo $R$ contains the block $(1, 2, \ldots, H)$.\\

We show how to use this result for proving Theorem \ref{mainthm} in Section \ref{smain}.

\section{General notation and results}\label{genlem}

\subsection{Notation}

For a real number $u$, we can write in a unique way $u = \lfloor u \rfloor + \{u\}$, where $\lfloor u \rfloor  \in \mathbb{Z}$ and $\{u\} \in [0, 1)$. We further let $\|u\| = \min (\{u\}, 1-\{u\}) = \min\{|u-m| \colon m \in \mathbb{Z}\} $.\\

For a $s$-tuple of real numbers $k=(k_1, \ldots, k_s)$, we let $\|k\|_{\infty}$ denote the maximum of the absolute values of its components. \\

For a real number $t$, we let $\e(t)=\exp(2\pi i t)$.\\

The letters $p$, $q$ denote prime numbers, the letter $h$ a non-negative integer.\\

We use Vinogradov's notation: if $f$ and $g$ are two real functions defined on some interval $[a, +\infty)$, $g$ taking positive values, we write $f \ll g$ for $f =O(g)$ and, when a parameter $u$ is involved, we write $f\ll_u g$ for $f = O_u(g)$. Similarly, we define $f \gg g$ (\emph{resp.}  $f \gg_u g$) if there exists a positive constant $C$, absolute, (\emph{resp.} which may depend on $u$) such that $f(x) \ge C g(x)$ on $[a, +\infty)$.\\

 For  $c$ real and $\ell$ a positive integer, we let $\binom{c}{\ell} = \frac{c(c-1)\cdots (c-\ell+1)}{\ell!}$; we further let $\binom{c}{0}=1$. By Taylor's expansion, we have, for $c >0, m>{~}0$ and $h \ge 0$
\begin{equation}\label{Taylor}
(m+h)^c = \sum_{\ell=0}^{\lfloor c \rfloor} \binom{c}{\ell} h^{\ell} m^{c-\ell} + r_c(m,h),
\end{equation}
with $0 \le r_c(m, h) \le \binom{c}{\lfloor c \rfloor +1} h^{\lfloor c \rfloor +1}m^{\{c\}-1} \le  h^{\lfloor c \rfloor +1}m^{\{c\}-1} $.

\subsection{The Erd\H{o}s-Tur\'{a}n-Koksma-Sz\"{u}sz inequality}\label{sETKS}

This inequality permits to give an upper bound for the discrepancy of a sequence of elements in $\mathbb{R}^s$ in terms of finitely many trigonometrical sums. The case when $s=1$ has been given by Erd\H{o}s and Tur\'{a}n in \cite{ET} and generalized to higher dimensions independently by Koksma \cite{K} and Sz\"{u}sz \cite{Sz}. We quote here the formulation given by Drmota and Tichy \cite{DT}. We first recall the definition of the discrepancy.\\

An \emph{interval} $I$ in $[0, 1)^s$ is a cartesian product $\prod_{1 \le i \le s} [a_i, b_i)$, with $0 \le a_i \le b_i \le 1$ for $1 \le i \le s$; its Lebesgue measure $\prod_{1 \le i \le s} (b_i-a_i)$ is denoted by $\lambda(I)$. For an interval $I \subset [0, 1)^s$, we denote by $\chi_I$ its indicator (also called characteristic) function.

The \emph{discrepancy} of a finite set $X = \{x_1, x_2, \ldots, x_N\}$ of elements of $\mathbb{R}^k$ is defined by
\begin{equation}\label{defdiscr}
D_N(X) = D_N(x_1, x_2, \ldots, x_N) = \sup_{I \subset [0, 1)^s} 
\abs{ \frac{1}{N} \sum_{n=1}^N \chi_I(\{x_n\}) - \lambda(I)}.
\end{equation}

\begin{lem}[Multidimensional Erd\H{o}s-Tur\'{a}n-Koksma-Sz\"{u}sz inequality]\label{ETKS}
Let $s \ge 1$, $X = \{x_1, x_2, \ldots, x_N\}$ be a finite set of elements of $\mathbb{R}^s$ and $K$ an arbitrary positive integer. We have
\begin{equation}\label{equETKS}
D_N(X)\le \left(\frac{3}{2}\right)^s\left(\frac{2}{K+1} + \sum_{0<\|k\|_{\infty}\le K} \frac{1}{r(k)}\abs{\frac{1}{N} \sum_{n=1}^N \e(k\cdot x_n)}\right),
\end{equation}
where, for $k = (k_1, k_2, \ldots, k_s) \in \mathbb{Z}^s$, we let $r(k) = \prod_{i=1}^s \max\{1, |k_i|\}$ and $u\cdot v$ denote the usual scalar product of two elements $u$ and $v$ in $\mathbb{R}^s$.
\end{lem}

\subsection{Upper bound for some trigonometrical sums}\label{strigosum}

The following lemma provides us with an upper bound for the trigonometric sums useful in the application of the ETKSz inequality. The exponent in (\ref{eup}) is not very strong, but the uniformity in the statement makes it very convenient for our purpose. It is a rather straightforward consequence of the van der Corput inequalities.

\begin{lem}\label{upper}
Let $c$ be a real number in $(1, \infty)\backslash \mathbb{N}$.
There exist two constants $A$ and $B$ such that for any $N$ larger than $A$ and any non-zero $(\ec+1)$-tuple $(a_0, \ldots, a_{\ec})$ satisfying
\begin{equation}\label{eup}
\forall \ell \in [0, \ec]  \colon \text{ either } a_{\ell} = 0 \text{ or } N^{-\dc/2} \le |a_{\ell}| \le N^{\dc/2},
\end{equation}
one has
$$
S:= \left|\sum_{n = N+1}^{2N} \operatorname{\e} \left(\sum_{\ell = 0}^{\ec} a_{\ell} n^{c - \ell}\right) \right| \le B N^{1-\theta},
$$
where $\theta = 2^{-(c+2)} \|c\|$.
\end{lem}

\begin{proof}
We start with the case when all the coefficients are $0$ except $a_{\ec}$. In this case, we use Kusmin-Landau inequality (cf. Theorem 2.1 of \cite{GK}). If $N$ is large enough, we have, on $(N, 2N]$, $\|\frac{d}{dt}(a_{\ec}t^{\fc})\|\gg N^{-\dc/2 + \fc-1}$ and thus $S \ll N^{1+\dc/2 - \fc} \ll N^{1-\dc/2} \le N^{1-\theta}$.\\

For the other cases, we use the following lemma, due to van der Corput, which one can find as a combination of Theorem 2.2 (for $q=0$) and Theorem 2.8 (for $q \ge 1$) of \cite{GK}.
\begin{lem}\label{vdC}
There exists an absolute constant $C$ satisfying the following. Let $q$ be a non-negative integer, $f$ a real valued function defined on $[N, 2N]$, with $q+2$ continuous derivatives satisfying
$$
\lambda \le |f^{(q+2)}(t)|\le \a \lambda,
$$
for some $\lambda >0$ and some $\a \ge 1$. Then, one has
\begin{equation}\label{evdC}
 \sum_{n=N+1}^{2N} \e(f(n)) \le C\left(N(\a^2\lambda)^{\frac{1}{4Q-2}} +N^{1-\frac{1}{2Q}}\a^{\frac{1}{2Q}}+N^{(\frac{Q-1}{Q})^2}\lambda^{\frac{-1}{2Q}}\right),
\end{equation}
where $Q=2^q$.
\end{lem}
Let $\ell_0$ be the smallest integer $\ell$ such that $a_{\ell}$ is non zero. The function of interest is thus $f(n) = \sum_{\ell = \ell_0}^{\ec} a_{\ell} n^{c - \ell}$. Since the non zero coefficients $a_{\ell}$, including $a_{\ell_0}$ are small powers of $N$, the term $a_{\ell_0}n^{c-\ell_0}$ dominates $f$ and its derivatives dominate the derivatives of $f$. In the application of Lemma \ref{vdC}, we take $q = \ec - \ell_0 -1$; we first notice that $q$ is non-negative (we already treated the case $\ell_0 = \ec$); we also notice that $q \le \ec-1 \le c-1$ so that $1 \le Q\le 2^{c-1}$. Furthermore, $\a$ depends only of $c$, and (since our constants may depend on $c$), it plays no role in our application of (\ref{evdC}); moreover $\lambda$ is of the order $|a_{\ell_0}|N^{\fc -1}$, so that
$$
\lambda \ll_c N^{\dc/2 + \fc -1} \ll_c N^{- \|c\|/2} \text{ and } \lambda^{-1} \ll_c N^{\dc/2 +1 -\fc} \ll_c N^{3 \|c\|/2}.
$$
By (\ref{evdC}), we have
$$
S\ll_c N^{1-\frac{ \|c\|}{8Q-4}} + N^{1-\frac{2}{Q}} +N^{(\frac{Q-1}{Q})^2+\frac{3\|c\|}{4Q}}.
$$
We easily notice that each of the three terms in RHS of the last relation is less than $N^{1-\theta}$.
\end{proof}

\subsection{Contribution of large primes to $\f(n)/n$}\label{slargeprimes}
\begin{lem}\label{largep}
Let $\a >0$. For any $C$ in $(0, \min(\a, 1))$ there exists $N$ such that for $n \ge N$, one has
\begin{equation}\label{elargep}
\prod_{\substack{p|n\\p\ge (\log n)^{\a}}} \left(1 - \frac{1}{p}\right) \ge C.
\end{equation}
\end{lem}

\begin{proof} 
This is a simple consequence of classical results on the distribution of prime numbers. Explicit value for $N$ in terms of $\a$ and $C$, which is not relevant here, can be obtained from the evaluations given in \cite{D}.\\
Since the LHS of (\ref{elargep}) is an increasing function of $\a$, it enough to consider the case when $\a \in (0,1)$. 
Let $n$ be sufficiently large an integer. We denote by $p_0, p_1, \ldots, p_r, p_{r+1}$ the consecutive prime numbers such that
\begin{equation}\label{consprime}
p_0 <\log^{\a} n\le p_1< \cdots <p_{r+1}\text{ and } p_1 \cdots p_r \le n < p_1 \cdots p_{r+1}.
\end{equation}
By a direct lower bound and by Mertens Theorem, for $n$ large enough, we have 
$$
\prod_{\substack{p|n\\p\ge (\log n)^{\a}}} \left(1 - \frac{1}{p}\right) \ge \prod_{j=1}^r \left(1 - \frac{1}{p_j}\right) = (1 + o(1))\frac{\log p_0}{\log p_r}.
$$
From the first part of (\ref{consprime}), we have $p_0 =(1+o(1)) \log^{\a} n$. Taking logarithms in the second part of (\ref{consprime}), we have $\log n =(1+o(1)) (\theta(p_r)-\theta(p_0))$. The last relations we obtained imply $p_r =(1 + o(1)) \log n$ and so we have
$$
\log p_0 / \log p_r =(1+o(1)) \a \log \log n / \log \log n = \a +o(1).
$$
Thus, for any $0<C < \min(\a,1)$, Relation (\ref{elargep}) holds true when $n$ is large enough.
\end{proof}

\subsection{Elementary relations concerning integral and fractional parts}
For convenience, we recall here some elementary relations.\\

Let $\ell$ denote a non-negative integer and $x, x_1,\ldots,x_{\ell}$ real numbers. We have
\begin{equation}\label{multfrac}
\lfloor \ell x\rfloor \ge \ell \lfloor x \rfloor \; \text{ and } \;\{\ell x\} \le \ell \{x\},
\end{equation}
\begin{equation}\label{sumint}
\text{if } \{x_1\}+\cdots +\{x_{\ell}\} <1, \; \text{ then } \; \lfloor x_1 + \cdots +x_{\ell} \rfloor= \lfloor x_1\rfloor  \cdots + \lfloor x_{\ell} \rfloor,
\end{equation}
\begin{equation}\label{mult}
\text{if } \; \ell \{x\} < 1, \; \text{ then }\; \lfloor \ell x\rfloor = \ell \lfloor x \rfloor.
\end{equation}
\textit{Proof of those relations.}
We multiply $x = \lfloor x \rfloor + \{x\}$ by $\ell$, which leads to 
\begin{alignat*}{2}
\ell x  = \,&\ell\lfloor x \rfloor + \ell \{x\}\\
= \,&\ell \lfloor x \rfloor + \lfloor \ell \{x\}\rfloor + \{\ell \{x\}\}
\end{alignat*}
which leads in turn to $\lfloor \ell x \rfloor = \ell \lfloor x \rfloor + \lfloor \ell \{x\}\rfloor $ and $\{\ell x\} = \{\ell\{x\}\}= \ell\{x\} - \lfloor \ell\{x\}\rfloor$, whence (\ref{multfrac}).\\
We have
$$
x_1+\cdots + x_{\ell} =( \lfloor x_1\rfloor + \cdots +  \lfloor x_{\ell}\rfloor) + (\{x_1\} + \cdots +\{x_{\ell}\});
$$
The first term in the RHS is an integer and the second term is in $[0, 1)$, which implies (\ref{sumint}). Relation (\ref{mult}) is just the special case of (\ref{sumint}) in which all the summands are equal.
\qed

\begin{lem}\label{congfrac}
Let $R$ and $r$ be integers with $0 \le r < R$, $u$ be a real number in $[0, 1]$, and $x$ be a real number such that
$$
\left\{\frac{x}{R}\right\} \in \left[\frac{r}{R}, \frac{r+u}{R}\right).
$$
We have
$$
\lfloor x \rfloor \equiv r\, (\mod R) \; \text{ and }\; \{x\} < u.
$$
\end{lem}
\begin{proof}
The hypothesis implies that there exists a rational integer $t$ such that
$$
\frac{x}{R} -t \in \left[\frac{r}{R}, \frac{r+u}{R}\right).
$$
This implies
$$
tR+r \le x < tR+r+u, \; \text{ whence } \; \lfloor x \rfloor = tR + r \; \text{ and } \; 0 \le \{x\} < u.
$$
\end{proof}

\section{Local approximation by a linear polynomial}\label{slinpol}

In this section, we keep the notation of Theorem \ref{slinpol}, assuming, without loss of generality that $0\le r < R$.

\subsection{A lemma on fractional parts}

\begin{lem}\label{fracgame}
Let $m$ be an integer satisfying
\begin{enumerate}[label=(\roman*)]
\item \quad $m^{1-\fc}> 4(cH)^{{c+1}}$,
\item \quad $\left\{\frac{m^c}{R}\right\} \in \left[\frac{r}{R},  \frac{r}{R} +\frac{1}{4R}\right),   $
\item \quad $\left\{\frac{cm^{c-1}}{R} \right\} \in \left[\frac{1}{R},\frac{1}{R}+\frac{1}{4RH}\right), $
\item \quad $\forall\ell  \in [2, \ec] \colon \left\{\frac{\binom{c}{\ell} m^{c-\ell}}{R}\right\} \in \left[0, \frac{1}{4cRH^{\ell}}\right)$.
\end{enumerate}
Then, for any $h$ in $[1, H]$, we have
$$
\lfloor (m+h)^c \rfloor \equiv r +h\; (\mod R).
$$
\end{lem}
\begin{proof}
Let $h$ be an integer in $[1, H]$. \\
By Relation ($i$), we have $0\le c^ch^{\ec+1}m^{\fc-1} < 1/4$, which implies $\{c^c h^{\ec+1}m^{\fc-1}\} < 1/4$.\\
By Lemma \ref{congfrac}, ($ii$) implies $\lfloor m^c\rfloor \equiv r \; (\mod R)$ and $\{m^c\} < 1/4$.\\
By Lemma \ref{congfrac}, ($iii$) implies $\lfloor cm^{c-1} \rfloor \equiv 1 \; (\mod R)$ and $\{cm^{c-1}\} <1/4H$ ; thanks to (\ref{multfrac}) and (\ref{mult}), we obtain $\lfloor h c m^{c-1} \rfloor \equiv h\lfloor c m^{c-1} \rfloor \equiv h \; (\mod R)$ and $h\{cm^{c-1}\}< 1/4$.\\
In a similar way, we obtain $\lfloor h \binom{c}{\ell} m^{c-\ell} \rfloor \equiv 0 \; (\mod R)$ as well as $h^{\ell}\{ \binom{c}{\ell} m^{c-\ell} \}<1/4c$.\\
Using Taylor's expansion (\ref{Taylor}) and (\ref{sumint}), we end the proof of Lemma{~}\ref{fracgame}.
\end{proof}

\subsection{The use of Erd\H{o}s-Tur\'{a}n-Koksma-Sz\"{u}sz inequailty}

In this subsection, we let $R$ be a (large) integer and
$$
\b = 2^{c+2}(c+2)^2/\dc, \; N = \lfloor R^{\b} /2 \rfloor,
$$
$$
s = \ec+1  \text{ and } x_n=\left(\frac{(N+n)^{c}}{R},\frac{c(N+n)^{c-1}}{R}, \ldots, \frac{\binom{c}{\ec}(N+n)^{\fc}}{R}\right).
$$
Our aim is to show the following
\begin{lem}\label{discrX}
As $R$ tends to infinity, we have 
\begin{equation}\label{ediscrX}
D_N(x_1, x_2, \ldots, x_N) = o_{c,H}\left(R^{c+1}\right).
\end{equation}
\end{lem}
We use ETKSz inequality, with
\begin{equation}\notag%\label{defK}
K = \lceil N^{\theta/(c+2)} \rceil, \text{ where $\theta = 2^{-(c+2)}\dc$ (cf. Lemma \ref{upper})}.
\end{equation}
Up to a slight change of notation (the indices of the components of $k$ running from $0$ to $\ec$ and the indices of the trigonometric sums running from $N+1$ to $2N$), the trigonometric sums we have to consider are
$$
S(N; k, R) = \sum_{n=N+1}^{2N}\e\left(R^{-1} \sum_{\ell = 0}^{\ec}k_{\ell}\binom{c}{\ell}n^{c-\ell}\right),
$$
which are of the type studied in Lemma \ref{upper}, with $a_{\ell} = R^{-1}k_{\ell}\binom{c}{\ell}$. We have $1 \le \binom{c}{\ell} \ll_c(1)$ and thus
\begin{equation}\notag%\label{al}
\text{either } a_{\ell} =0 \text{ or } R^{-1} \le |a_{\ell} | \ll_c R^{-1} K.
\end{equation}
When $a_\ell \neq 0$, we have 
$$
|a_{\ell}| \ll_c R^{-1}K \ll_c K \ll_c N^{\theta/(c+2)} \; \text{ and }\; |a_{\ell}^{-1}| \le R \ll_c N^{\theta/(c+2)^2}.
$$
If $R$ is sufficiently large, we can use Lemma \ref{upper} and obtain
$$
\forall k \text{ with } 0 <  \|k\|_{\infty} \le K \, \colon \, S(N;k,R) \le N^{1-\theta}.
$$
We can now apply Lemma \ref{ETKS}, with the most trivial bound $r(k) \ge{~}1$. We obtain, when $R$, or equivalently $N$, is large enough
$$
D_N(X) \ll_c K^{-1} + K^{\ec +1}N^{-\theta} \ll_c  N^{-\theta/(c+2)}.
$$

\subsection{End of the proof of Theorem \ref{tlinpol}}
We first prove that, when $R$ is sufficiently large, there exists an integer $m \in [N+1, 2N]$ which satisfies Relations ($ii$), ($iii$) and ($iv$) of Lemma \ref{fracgame}. For doing so, we wish to find an element $x_m$ in $X$ which belongs to an interval $I$ of Lebesgue measure
$$
\lambda(I) = \frac{1}{4R} \times \frac{1}{4R} \times \prod_{\ell=2}^{\ec} \frac{1}{4cRH^{\ell}} \gg_{c, H} R^{-\ec -1}.
$$
By the definition of the discrepancy recalled in (\ref{defdiscr}), any interval in $[0, 1)^s$ with Lebesgue measure larger than $D_N(X)$ contains an element from $X$. We have
$$
(\lambda(I))^{-1} \ll_{c,H} R^{{\ec+1}} \ll_{c,H} N^{\theta (\ec+1)/(c+2)^2}  =o_{c,H}(D_N(X)).
$$
Thus, if $N$ is large enough, there exists $m$ in $[N+1, 2N]$ which satisfies Relations ($ii$), ($iii$) and ($iv$) of Lemma \ref{fracgame}. To complete the proof of Theorem \ref{tlinpol}, it is enough to notice that when $N$ is large enough, the integer $m$ we have produced satisfies also Relation ($i$) and to apply Lemma \ref{fracgame}.
$\qed$

\section{Proof of Theorem \ref{mainthm}}\label{smain}
We follow a similar strategy to that of \cite{DL}: for an integer $H$ arbitrary large, we construct an integer $m$ such that
\begin{equation}\label{fundam}
\forall h \in [1, H] \;\colon \; \frac{1}{H} \le \frac{\f(\lfloor (m+h)^c \rfloor)}{\lfloor (m+h)^c \rfloor)} \le \frac{3}{H}.
\end{equation}
Theorem \ref{mainthm} easily follows from (\ref{fundam}).\\

Without loss of generality, we may assume that $H \ge 21$ which implies that for any $h$ in $[1, H]$, one has $\frac{\f(h)}{h} > \frac{3}{H}$ (notice that $\f(h)$ is always at least $1$ and is larger than $3$ for $h \ge H/3\ge 7$. Since  $\f(p)/p$ tends to $1$ when the prime $p$ tends to infinity and the infinite product of those terms diverges to $0$ (by Euler), we can find finite families of prime numbers $\left(\mathcal{P}_h\right)_n$ such that
\begin{enumerate}[label=(\roman*)]
\item for $1 \le h < k \le H$ the sets $\mathcal{P}_h$ and $\mathcal{P}_k$ are disjoint,
\item for any $h$ all the elements of $\mathcal{P}_h$ are larger than $H$,
\item for any $h \; \colon \; 2/H  \le (\f(h)/h)(\f(P_h)/P_h) \le 3/H$,
\end{enumerate}
where $P_h$ denotes the product of the primes in $\mathcal{P}_h$.\\

We now let $L$ be a prime number which is larger than all the elements of all the $\mathcal{P}_h$ and we let
$$
R= H! \prod_{H<p\le L} p.
$$
By the Chinese Remainder Theorem, we can find a residue $r$ modulo $R$ satisfying
\begin{eqnarray*}
&r \equiv -1 \; (\mod P_1)\\
&\cdots \\
&r \equiv -H \; (\mod P_H)\\
&r \equiv 0 \; (\mod R/(P_1P_2 \cdots P_H)).
\end{eqnarray*}

By Theorem \ref{tlinpol}, if $R$ is large enough, i.e. if $L$ is large enough, there exists $m \le R^{\b}$ such that 
$$
\forall h \in [1,H] \; \colon \; \lfloor (m+h)^c\rfloor \equiv r+h \; (\mod R)
$$
and thus $\gcd(\lfloor (m+h)^c \rfloor, R) = hP_h$. In other words, $\lfloor (m+h)^c \rfloor $ is an integer which is less than $R^{\b c +1}$, divisible by $hP_h$ and such that any of its prime factor which does not divide $hP_h$ is larger than $L$. We thus have, using Relation (iii) above
$$
\frac{2}{H} \prod_{\substack{p |\lfloor (m+h)^c \rfloor \\ p > L}}  \frac{\f(p)}{p}\le \f(\lfloor (m+h)^c\rfloor)/\lfloor (m+h)^c \rfloor) \le \frac{3}{H}.
$$
In order to prove (\ref{fundam}), and thus Theorem  \ref{mainthm}, it is enough to prove that for $L$ large enough, we have

\begin{equation}\label{ouf}
 \prod_{\substack{p |\lfloor (m+h)^c \rfloor \\ p > L}}  \left(1-\frac{1}{p}\right) \ge 1/2.
\end{equation}

By the definition of $R$ and the prime number theorem, we have, when $L$ is large enough
\begin{align*}
\log \lfloor (m+h)^c \rfloor &\le (\b c +1) \left(\sum_{p \le L} \log p+ \log H!\right) \\
&\le (\b c+2) L.
\end{align*}
When $L$ is large enough, and thus $m$ too, we have
\begin{alignat*}{2}
 \prod_{\substack{p |\lfloor (m+h)^c \rfloor \\ p > L}}  \left(1-\frac{1}{p}\right) &\ge  \prod_{\substack{p |\lfloor (m+h)^c \rfloor \\ p >((\b c +2)^{-1} \log\lfloor (m+h)^c \rfloor }}  \left(1-\frac{1}{p}\right) &\text{by the previous inequality}\\
 &\ge  \prod_{\substack{p |\lfloor (m+h)^c \rfloor \\ p >(\log\lfloor (m+h)^c \rfloor)^{3/4} }}  \left(1-\frac{1}{p}\right)  &\text{since $m$ is large enough}\\
 &\ge 1/2 &\text{by Lemma \ref{largep}}.
\end{alignat*}
Thus, (\ref{ouf}) is proved, as well as Theorem \ref{mainthm}.

\section{Remarks and complements}

We see this paper as an archetype of the more general study of the distribution modulo 1 of mean vaues of sequences $(f(a(n)))_n$, where $f$ is a regular multiplicative function with constant mean value and $a$ is a function in a Hardy field with polynomial growth.\\

In particular, it is challenging to determine the values of $c$ for which the sequence $\left(\sum_{m\leq n} \f(\lfloor m^c \rfloor) / \lfloor m^c \rfloor \right)_{n\ge 1}$ is uniformly distributed modulo $1$.

\section{Acknowledgement}

The first-named author has benefitted from the support of the joint FWF-ANR project Arithrand: FWF: I 4945-N and ANR-20-CE91-0006.\\

\end{document}